\documentclass[11pt]{article}
\usepackage[margin=1in]{geometry}
\usepackage{amsmath,amssymb,amsthm}
\usepackage{mathtools}
\usepackage{enumitem}
\usepackage[colorlinks=true,linkcolor=blue,citecolor=blue]{hyperref}

\theoremstyle{plain}
\newtheorem{theorem}{Theorem}
\newtheorem{proposition}[theorem]{Proposition}
\newtheorem{lemma}[theorem]{Lemma}
\newtheorem{corollary}[theorem]{Corollary}
\theoremstyle{definition}
\newtheorem{definition}[theorem]{Definition}
\newtheorem{remark}[theorem]{Remark}
\newtheorem{convention}[theorem]{Convention}

\newcommand{\Aut}{\operatorname{Aut}}
\newcommand{\Sym}{\operatorname{Sym}}
\newcommand{\Stab}{\operatorname{Stab}}
\newcommand{\Z}{\mathbb{Z}}
\newcommand{\Sig}{\Sigma_{Q}}
\newcommand{\ord}{\operatorname{ord}}
\newcommand{\id}{\operatorname{id}}

\title{Minimum Size of a Poset Realizing $\Z_{2}\times\Z_{2^{n}}$\\ as its Automorphism Group}
\author{Ponaki Das\thanks{Department of Mathematics, North-Eastern Hill University,
NEHU Campus, Shillong 793022, India. \texttt{ponaki.das20@gmail.com}}
\and
Sainkupar Marwein Mawiong\thanks{Department of Basic Sciences and Social Sciences,
North-Eastern Hill University, NEHU Campus, Shillong 793022, India.
\texttt{skupar@gmail.com}. Corresponding author.}}
\date{}

\begin{document}
\maketitle

\begin{abstract}
We study the realization of finite groups as automorphism groups of finite
posets. Given a finite group $G$, let $\beta(G)$ denote the smallest number of
elements in a poset $P$ with $\Aut(P)\cong G$. While $\beta(G)$ is known for
several cyclic and small abelian groups, the non-cyclic abelian case is largely
open. In this paper we prove that
$\beta(\Z_{2}\times\Z_{2^{n}})=2^{\,n+1}+2$ for every $n\ge 3$.
\end{abstract}


\section{Introduction}

The study of finite topological spaces, once a curiosity, has become a tool in
algebraic topology, largely through the work of McCord~\cite{McCord}, who showed
that every finite $T_{0}$ space is weak homotopy equivalent to the geometric
realization of its order complex. This lets one model topological phenomena by
finite posets, and the automorphism group $\Aut(P)$ of a poset $P$ is isomorphic
to the homeomorphism group of the corresponding finite space. It is therefore
natural to ask, for a finite group $G$, what is the smallest poset realizing $G$
as its full symmetry group. To measure this we use the invariant
\[
\beta(G):=\min\{\,|P| : \Aut(P)\cong G\,\}.
\]
Combinatorially $\beta(G)$
quantifies the minimality of a poset with prescribed symmetry; topologically it
bounds below the cardinality of any finite $T_{0}$ space with homeomorphism
group $G$.

Minimal realizations of abelian groups by graphs were studied by Coxeter and
others~\cite{Arlinghaus,Coxeter,Meriwether}; in the digraph setting Babai
initiated the study of regular automorphism groups~\cite{Babai1,Babai2}, and in
the poset setting Barmak and Barreto developed a systematic theory for abelian
groups~\cite{Barmak2,Barreto,BarmakBarreto}. In this paper we determine
$\beta(G)$ for the infinite family $G=\Z_{2}\times\Z_{2^{n}}$, $n\ge 3$:
\[
\beta(\Z_{2}\times\Z_{2^{n}})=2^{\,n+1}+2,\qquad n\ge 3 .
\]
We obtain a sharp lower bound by analyzing the kernel of the action on free
orbits and showing that any minimal poset must contain at least two free orbits
and at least two non-free points; the matching upper bound comes from an
explicit, self-contained construction that adjoins a $\Z_{2}$ factor to a
minimal realization of $\Z_{2^{n}}$. The boundary case $n=2$ is genuinely
different: there $\beta(\Z_{4})=12$ rather than $2\cdot2^{2}$
(Corollary~\ref{cor:cyclic}, since $b(4)=3$), so already the upper bound moves,
giving $\beta(\Z_{2}\times\Z_{4})\le14\neq2^{\,3}+2$; the matching lower bound,
and hence the exact value, is established separately
in~\cite{DasMawiongZ4Pre}. The case $n=1$ concerns $\Z_{2}\times\Z_{2}$, where
$\beta(\Z_{2}\times\Z_{2})=4$, and is classical (see~\cite{Gyenizse,BarmakBarreto}).
These boundary cases are recorded
only to delimit the scope of our result; they play no role in the proof of
Theorem~\ref{thm:main}, which assumes $n\ge3$ throughout.

\section{Preliminaries}

We recall the bounds for direct products and cyclic groups that we need, and we
fix terminology and prove the structural lemmas used throughout.

All posets are finite and carry a (non-strict) partial order $\le$, i.e.\ a
reflexive, antisymmetric, transitive relation. We write $x<y$ for $x\le y$ with
$x\neq y$; the strict part $<$ is then irreflexive and transitive. We write
$x\parallel y$, and call $x,y$ \emph{incomparable}, when neither $x\le y$ nor
$y\le x$; thus for any $x\neq y$ exactly one of $x<y$, $y<x$, $x\parallel y$
holds. An \emph{antichain} is a set of pairwise incomparable points.

\begin{convention}\label{conv:realize}
We say that a finite poset $P$ \emph{realizes} a finite group $G$ if
$\Aut(P)\cong G$. We fix such an isomorphism and identify $G$ with $\Aut(P)$;
thus $G$ acts on $P$ by order-automorphisms, this action is faithful, and its
image is all of $\Aut(P)$. The poset $P$ is \emph{minimal} (for $G$) if, in
addition, $|P|=\beta(G)$. Throughout, $G=H\times K$ with $H=\langle h\rangle\cong
\Z_{2^{n}}$, $K=\langle k\rangle\cong\Z_{2}$, $n\ge 3$, and $m:=2^{n}$ (the
order of $H$).
\end{convention}

The only direct-product fact needed in this paper is the special case of
multiplication by $\Z_{2}$, for which we give a short self-contained proof
(thus the upper bound below depends on no unpublished source).

\begin{proposition}\label{prop:timesZ2}
Let $A$ be a finite group and let $P_{1}$ be a finite poset with
$\Aut(P_{1})\cong A$. Pick two elements $a_{1},a_{2}\notin P_{1}$ and let
$P:=P_{1}\cup\{a_{1},a_{2}\}$ be ordered by: the order of $P_{1}$ on $P_{1}$;
$a_{1}$ and $a_{2}$ mutually incomparable; $a_{i}<p$ for every $p\in P_{1}$ and
$i\in\{1,2\}$; and no further relations. Then $P$ is a poset,
$|P|=|P_{1}|+2$, and
\[
\Aut(P)\;\cong\;A\times\Z_{2}.
\]
In particular $\beta(A\times\Z_{2})\le\beta(A)+2$.
\end{proposition}

\begin{proof}
\emph{$P$ is a poset.} Reflexivity is built in. For antisymmetry, the only
relations are within $P_{1}$ (a poset) and the relations $a_{i}<p$
($p\in P_{1}$); no relation has a $P_{1}$-element below an $a_{i}$, and
$a_{1},a_{2}$ are incomparable, so no two distinct elements are mutually
related. For transitivity, the only chains through a new element are
$a_{i}\le a_{i}\le p$ and $a_{i}\le p\le q$ with $p,q\in P_{1}$; in the second,
$a_{i}<q$ holds by definition, and chains entirely in $P_{1}$ are transitive
there. No other chains through a new element occur: nothing lies strictly below
an $a_{i}$, and the only elements above an $a_{i}$ are points of $P_{1}$, so an
$a_{i}$ never occurs as the interior or top vertex of a chain. Hence $\le$ is a
partial order, and clearly $|P|=|P_{1}|+2$.

\emph{The minimal elements of $P$ are exactly $a_{1},a_{2}$.} Nothing lies below
$a_{1}$ or below $a_{2}$, so both are minimal. For any $p\in P_{1}$ we have
$a_{1}<p$, so $p$ is not minimal. Thus $\min(P)=\{a_{1},a_{2}\}$, a
set preserved by every automorphism. Consequently every $\varphi\in\Aut(P)$
restricts to a bijection of $\{a_{1},a_{2}\}$ and to a bijection of
$P\setminus\{a_{1},a_{2}\}=P_{1}$.

Define
\[
\Psi:\Aut(P)\longrightarrow\Aut(P_{1})\times\Sym\{a_{1},a_{2}\},\qquad
\Psi(\varphi)=\bigl(\varphi|_{P_{1}},\ \varphi|_{\{a_{1},a_{2}\}}\bigr).
\]

\emph{$\Psi$ is well defined.} For $p,q\in P_{1}$, the restriction of $\le_{P}$
to $P_{1}$ is exactly $\le_{P_{1}}$; since $\varphi$ preserves $\le_{P}$ and
maps $P_{1}$ onto $P_{1}$, $\varphi|_{P_{1}}$ preserves $\le_{P_{1}}$ and is a
bijection, hence $\varphi|_{P_{1}}\in\Aut(P_{1})$. Also
$\varphi|_{\{a_{1},a_{2}\}}\in\Sym\{a_{1},a_{2}\}$. Since both $P_{1}$ and
$\{a_{1},a_{2}\}$ are invariant under every automorphism, $\Psi$ is a group
homomorphism. It is injective because $P=P_{1}\sqcup\{a_{1},a_{2}\}$, so
$\varphi$ is determined by its two restrictions.

\emph{$\Psi$ is surjective.} Given $\alpha\in\Aut(P_{1})$ and
$\varepsilon\in\Sym\{a_{1},a_{2}\}$, define $\varphi:=\alpha$ on $P_{1}$ and
$\varphi:=\varepsilon$ on $\{a_{1},a_{2}\}$; this is a bijection of $P$. We
check $x\le_{P}y\iff\varphi(x)\le_{P}\varphi(y)$ in all cases. If
$x,y\in P_{1}$, this is $x\le_{P_{1}}y\iff\alpha(x)\le_{P_{1}}\alpha(y)$, true
since $\alpha\in\Aut(P_{1})$. If $x=a_{i}$, $y\in P_{1}$, both sides hold (every
$a$-element is below every $P_{1}$-element, and $\varphi(x)\in\{a_{1},a_{2}\}$,
$\varphi(y)\in P_{1}$). If $x\in P_{1}$, $y=a_{j}$, both sides fail (no
$P_{1}$-element is $\le$ an $a$-element). If $x,y\in\{a_{1},a_{2}\}$, then
$x\le_{P}y\iff x=y\iff\varphi(x)=\varphi(y)\iff\varphi(x)\le_{P}\varphi(y)$,
since $a_{1}\parallel a_{2}$ and $\varepsilon$ is a bijection of a $2$-element
set. Hence $\varphi\in\Aut(P)$ and $\Psi(\varphi)=(\alpha,\varepsilon)$.

Therefore $\Psi$ is an isomorphism and
$\Aut(P)\cong\Aut(P_{1})\times\Sym\{a_{1},a_{2}\}\cong A\times\Z_{2}$. Taking
$P_{1}$ with $|P_{1}|=\beta(A)$ gives $\beta(A\times\Z_{2})\le\beta(A)+2$.
\end{proof}

\begin{corollary}[{\cite[Cor.~4.2]{BarmakBarreto}}]\label{cor:cyclic}
Let $N=p_{1}^{r_{1}}\cdots p_{k}^{r_{k}}$ with the $p_{i}$ pairwise distinct
primes and $r_{i}\ge 1$. Then the minimum number $\beta(\Z_{N})$ of points in a
poset with cyclic automorphism group of order $N$ is
\[
\sum_{i=1}^{k} b\!\left(p_{i}^{r_{i}}\right)\,p_{i}^{r_{i}}\;-\;1
\quad\text{if }\;3\mid N,\ 4\mid N,\ 9\nmid N,\ 8\nmid N
\quad(\text{equivalently, }3\,\|\,N\text{ and }4\,\|\,N),
\]
and is $\displaystyle\sum_{i=1}^{k} b\!\left(p_{i}^{r_{i}}\right)\,p_{i}^{r_{i}}$
otherwise, where
\[
b(2)=1,\qquad b(3)=b(4)=b(5)=b(7)=3,\qquad
b(p^{r})=2\ \text{for every other prime power.}
\]
\end{corollary}

The same value was obtained independently by Gyenizse, Hajnal and
Z\'adori~\cite{Gyenizse}.

\begin{remark}
For coprime $a,b$ the value is immediate: $\Z_{a}\times\Z_{b}\cong\Z_{ab}$ is
cyclic, so $\beta(\Z_{a}\times\Z_{b})=\beta(\Z_{ab})$ is read off directly from
Corollary~\ref{cor:cyclic}. The genuine difficulty in the abelian case is
therefore concentrated in the \emph{non-coprime} situation, where no such
reduction is available; $\Z_{2}\times\Z_{2^{n}}$ (with $\gcd(2,2^{n})=2$) is the
first infinite family of this kind, and is the subject of the present paper.
(In Corollary~\ref{cor:cyclic}, $N$ denotes a group order and should not be
confused with the exponent $n$ in $\Z_{2^{n}}$ used throughout the rest of the
paper.)
\end{remark}

\begin{lemma}\label{lem:antichain}
Let a finite group $\Gamma$ act on a finite poset $P$ by order-automorphisms.
Then every $\Gamma$-orbit is an antichain in $P$.
\end{lemma}

\begin{proof}
Suppose some orbit is not an antichain. Then there are distinct $x,y$ in one
orbit with $x<y$, say $y=\gamma\cdot x$ with $\gamma\neq e$. Since $\gamma$ is an
order-automorphism, $a<b\iff \gamma\cdot a<\gamma\cdot b$; applying $\gamma$
repeatedly to $x<\gamma\cdot x$ gives
\[
x<\gamma\cdot x<\gamma^{2}\cdot x<\cdots .
\]
Let $t=\ord(\gamma)$ (finite, as $\Gamma$ is finite). Then $\gamma^{t}=e$, so
\[
x<\gamma\cdot x<\cdots<\gamma^{t-1}\cdot x<\gamma^{t}\cdot x=x.
\]
The first step gives $x<\gamma\cdot x$, hence $x\le\gamma\cdot x$ with
$x\neq\gamma\cdot x$; the remaining steps give $\gamma\cdot x<x$ by transitivity
of $<$, hence $\gamma\cdot x\le x$. From $x\le\gamma\cdot x$ and
$\gamma\cdot x\le x$, antisymmetry forces $x=\gamma\cdot x$, contradicting
$x\neq\gamma\cdot x$. (Equivalently, composing the chain by transitivity yields
$x<x$, impossible since $<$ is irreflexive.) Thus every orbit is an antichain.
\end{proof}

\begin{remark}
No freeness hypothesis is required: the proof uses only that the action is by
order-automorphisms, so Lemma~\ref{lem:antichain} applies to every orbit (free
or not) of the $H$- and $G$-actions on $P$. This is the form used below, where
$P$ also contains non-free orbits.
\end{remark}

\begin{definition}\label{def:freeorbit}
An $H$-orbit $O\subseteq P$ is \emph{free} if $\Stab_{H}(x)=\{e\}$ for some
$x\in O$ (equivalently, for every $x\in O$: points of one $H$-orbit have
conjugate $H$-stabilizers, hence equal ones since $H$ is abelian), and
\emph{non-free} otherwise.
\end{definition}

\section{The realization \texorpdfstring{$\Z_{2}\times\Z_{2^{n}}$}{Z2 x Z2n}}

We first record the upper bound and then develop the signature machinery used
for the lower bound.

\begin{lemma}\label{lem:upper}
For $n\ge 3$, $\ \beta(\Z_{2}\times\Z_{2^{n}})\le 2^{\,n+1}+2$.
\end{lemma}

\begin{proof}
By Corollary~\ref{cor:cyclic} with $N=2^{n}$ ($n\ge3$): since $3\nmid 2^{n}$ the
``$-1$'' branch (which requires $3\mid N$) does not occur, and since
$2^{n}\notin\{2,4\}$ we have $b(2^{n})=2$, hence
$\beta(\Z_{2^{n}})=b(2^{n})\cdot 2^{n}=2\cdot 2^{n}=2^{\,n+1}$ (the same value
was obtained independently in~\cite{Gyenizse}). Let $P_{1}$ be a poset with
$\Aut(P_{1})\cong\Z_{2^{n}}$ and $|P_{1}|=\beta(\Z_{2^{n}})=2^{\,n+1}$. By
Proposition~\ref{prop:timesZ2} (applied with $A=\Z_{2^{n}}$, and using
$\Z_{2}\times\Z_{2^{n}}\cong\Z_{2^{n}}\times\Z_{2}$),
$\beta(\Z_{2}\times\Z_{2^{n}})=\beta(\Z_{2^{n}}\times\Z_{2})\le
|P_{1}|+2=2^{\,n+1}+2$.
\end{proof}

\begin{remark}
The lower bound $|P|\ge2^{\,n+1}+2$ (Corollary~\ref{cor:lower}) combines two
counts: $|R|\ge2m$ and $|Q|\ge2$. These rest on a chain of results, all
established in the remainder of this section; the following is a reading guide
to that chain, with forward references resolved below. The rigidity of free orbits
(Lemma~\ref{lem:transposition} $\to$ Lemma~\ref{lem:fixrigid} $\to$
Corollary~\ref{cor:sigrigid}) shows distinct points of a free $H$-orbit have
distinct full row-signatures. Combined with the antichain property
(Lemma~\ref{lem:antichain}), the kernel analysis on $G$-invariant free orbits
(Lemma~\ref{lem:kernel}), the dichotomy for $k$ on free orbits
(Lemma~\ref{lem:korbit}), and the two ``packaging'' lemmas that turn a swapping
automorphism or a $G$-fixed point into a uniform statement
(Lemmas~\ref{lem:rowsig} and~\ref{lem:uniform}), this yields $|Q|\ge2$
(Lemma~\ref{lem:Qge2}). Finally the transport conjugacy
(Lemma~\ref{lem:conj}) feeds Proposition~\ref{prop:notsingle}, which rules out a
single free orbit and so gives $|R|\ge2m$. Each contradiction is reached in one
of the three modes \textbf{(R)}, \textbf{(A)}, \textbf{(S)} catalogued before
Lemma~\ref{lem:Qge2}.
\end{remark}

\begin{definition}\label{def:sig}
For a finite poset $P$, two points $x,x'\in P$ have \emph{equal full
row-signatures} if, for every $y\in P$,
\[
x<y\iff x'<y\qquad\text{and}\qquad y<x\iff y<x' .
\]
(Taking $y\in\{x,x'\}$ shows $x\parallel x'$ when $x\neq x'$; and then
$x\parallel y\iff x'\parallel y$ for every $y$, incomparability being the
negation of the two comparabilities.) They have \emph{distinct} full
row-signatures when these conditions do not all hold.

Now let $R$ be the union of the free $H$-orbits of $P$ and $Q:=P\setminus R$ the
union of the non-free $H$-orbits. For an automorphism $a$ that preserves $Q$
(every automorphism does; see Remark~\ref{rem:invariance}), let
$a_{Q}\in\Sym(Q)$ be the permutation it induces on $Q$. The \emph{$Q$-signature
map} is $\Sig:R\to\{<,>,\parallel\}^{Q}$, where for $x\in R$ (so $x\notin Q$) and
$q\in Q$, $\Sig(x)(q)$ denotes the relation of $x$ to $q$ (the restriction to
$Q$ of $x$'s relations to $P$). We call any element of
$\{<,>,\parallel\}^{Q}$ a \emph{signature}, denoted by a bare $\Sigma$; thus
$\Sig(x)$ is the particular signature carried by the point $x$, while $\Sigma$
ranges over all signatures. We let $a_{Q}$ act on a signature
$\Sigma\in\{<,>,\parallel\}^{Q}$ by
\[
\bigl(a_{Q}\Sigma\bigr)(q)\;:=\;\Sigma\bigl(a^{-1}\cdot q\bigr),\qquad q\in Q ,
\]
i.e.\ by precomposition $\Sigma\mapsto\Sigma\circ(a_{Q})^{-1}$ (using
$(a_{Q})^{-1}(q)=a^{-1}\cdot q$ for $q\in Q$). This defines a group
homomorphism
\[
\Phi:\Sym(Q)\longrightarrow\Sym\bigl(\{<,>,\parallel\}^{Q}\bigr),\qquad
\Phi(\pi):\Sigma\mapsto\Sigma\circ\pi^{-1},
\]
and ``$a_{Q}$ acting on signatures'' means $\Phi(a_{Q})$; for a
$\Phi(a_{Q})$-invariant subset $S$ we write
$a_{Q}|_{S}:=\Phi(a_{Q})|_{S}\in\Sym(S)$. The \emph{transport equation} is
$\Sig(a\cdot x)=\Phi(a_{Q})\bigl(\Sig(x)\bigr)$. Three distinct permutation
actions thus appear and are kept notationally separate throughout: the action
$a_{Q}$ on the set $Q$; its image $\Phi(a_{Q})$ on the signature space
$\{<,>,\parallel\}^{Q}$; and the restriction $\Phi(a_{Q})|_{S}=a_{Q}|_{S}$ to a
$\Phi(a_{Q})$-invariant subset $S$ of that space.
\end{definition}

\begin{remark}\label{rem:invariance}
Since $G$ is abelian, $\Stab_{H}(g\cdot x)=\Stab_{H}(x)$ for $g\in G$, $x\in P$
(indeed $h'\cdot(g\cdot x)=g\cdot x\iff g\cdot(h'\cdot x)=g\cdot x\iff
h'\cdot x=x$). Hence ``$x$ lies in a free $H$-orbit'' is $G$-invariant, so $R$
and $Q$ are $G$-invariant. As $\Aut(P)=G$ (Convention~\ref{conv:realize}),
every automorphism preserves $Q$, so $a_{Q}$ is defined for all
$a\in\Aut(P)$. For an automorphism $a$ and $q\in Q$ one has
$a\cdot x\mathrel{\rho}q\iff x\mathrel{\rho}a^{-1}\cdot q$ for each
$\rho\in\{<,>,\parallel\}$; that is exactly
$\Sig(a\cdot x)(q)=\Sig(x)(a^{-1}\cdot q)=\bigl(\Phi(a_{Q})\Sig(x)\bigr)(q)$, the
transport equation.
\end{remark}

\begin{lemma}\label{lem:mcycle}
Let $c$ be an $m$-cycle in $\Sym(X)$, $|X|=m$.
\begin{enumerate}[label=\textup{(\arabic*)}]
\item For $0<j<m$, the power $c^{j}$ has no fixed point.
\item If $m=2^{n}$, the unique element of order $2$ in $\langle c\rangle$ is
$c^{m/2}$.
\end{enumerate}
\end{lemma}

\begin{proof}
$c^{j}$ is a product of $\gcd(m,j)$ cycles, each of length $m/\gcd(m,j)$. For
$0<j<m$ we have $\gcd(m,j)<m$, so every cycle has length $\ge2$ and $c^{j}$
fixes nothing; this is~(1). For~(2), $\langle c\rangle$ is cyclic of order
$2^{n}$, hence has a unique subgroup of order $2$, generated by its unique
element of order $2$, namely $c^{m/2}$.
\end{proof}

\begin{lemma}\label{lem:transposition}
Let $P$ be a finite poset and let $x\neq x'\in P$ have equal full
row-signatures (Definition~\ref{def:sig}). Then the
transposition $\tau=(x\;x')$ is an automorphism of $P$.
\end{lemma}

\begin{proof}
Taking $y=x'$ in $x<y\iff x'<y$ gives $x<x'\iff x'<x'$; since $x'<x'$ is false,
$x\not<x'$, and symmetrically $x'\not<x$. Hence $x\parallel x'$. Define
$\varphi$ by $\varphi(x)=x'$, $\varphi(x')=x$ and $\varphi(y)=y$ otherwise; it
is an involutive bijection. Let $u\le v$. If $\{u,v\}\cap\{x,x'\}=\varnothing$,
then $\varphi(u)=u\le v=\varphi(v)$. If $u=x$, $v\notin\{x,x'\}$, then
$x<v$, so $x'<v$ by signature equality, i.e.\ $\varphi(u)=x'\le v=\varphi(v)$
(the case $u=x'$ is symmetric). If $u\notin\{x,x'\}$, $v=x$, then $u<x$, so
$u<x'$, i.e.\ $\varphi(u)=u\le x'=\varphi(v)$ (the case $v=x'$ is symmetric).
The case $u,v\in\{x,x'\}$ can only be $u=v$ (as $x\parallel x'$), trivially
preserved. Thus $\varphi$ preserves $\le$; being an involution, it is an
automorphism.
\end{proof}

\begin{lemma}\label{lem:freeexists}
If $P$ realizes $G=\Z_{2}\times\Z_{2^{n}}$, then $R\neq\varnothing$.
\end{lemma}

\begin{proof}
The subgroups of $H\cong\Z_{2^{n}}$ form a single chain, with unique minimal
non-trivial subgroup $\langle h^{m/2}\rangle$ of order $2$ ($m/2=2^{n-1}$). If
no $H$-orbit were free, then $\Stab_{H}(x)\neq\{e\}$ for every $x\in P$, so each
$\Stab_{H}(x)$ contains $\langle h^{m/2}\rangle$. Then $h^{m/2}$ fixes every
point of $P$, i.e.\ acts as the identity automorphism, while $h^{m/2}\neq e$ in
$G$; this contradicts faithfulness (Convention~\ref{conv:realize}).
\end{proof}

The next two results isolate the rigidity of free orbits; they replace the
repeated ``construct an extra automorphism'' computations.

\begin{lemma}\label{lem:fixrigid}
Let $O$ be a free $H$-orbit and $\varphi\in G$. If $\varphi$ fixes some point
of $O$, then $\varphi$ fixes $O$ pointwise.
\end{lemma}

\begin{proof}
Say $\varphi(x_{0})=x_{0}$ with $x_{0}\in O$. As $O$ is a free $H$-orbit, every
$x\in O$ is $x=h^{j}\cdot x_{0}$ for some $j$. Since $G$ is abelian $\varphi$
commutes with $h$, so $\varphi(x)=h^{j}\cdot\varphi(x_{0})=h^{j}\cdot
x_{0}=x$.
\end{proof}

\begin{corollary}\label{cor:sigrigid}
Let $O$ be a free $H$-orbit of $P$. Then distinct points of $O$ have distinct
full row-signatures.
\end{corollary}

\begin{proof}
If distinct $x,x'\in O$ had equal full row-signatures, then by
Lemma~\ref{lem:transposition} the transposition $\tau=(x\;x')\in\Aut(P)=G$.
Since $|O|=m=2^{n}\ge8$, the set $O\setminus\{x,x'\}$ is nonempty (a third point
suffices); pick
$x''\in O\setminus\{x,x'\}$, which $\tau$ fixes. By Lemma~\ref{lem:fixrigid},
$\tau$ fixes $O$ pointwise, contradicting $\tau(x)=x'\neq x$.
\end{proof}

\begin{lemma}\label{lem:kernel}
Let $O$ be a free $H$-orbit that is invariant under $G$, and let
$\rho:G\to\Sym(O)$ be the restriction, $N:=\ker\rho$. Then $|N|=2$, and
$N=\langle g\rangle$ with $g\in\{k,\ h^{m/2}k\}$ (in particular $g\notin H$ and
$g$ fixes $O$ pointwise). Moreover $G\setminus H$ has exactly the two
involutions $k$ and $h^{m/2}k$, exactly one of which generates $N$; the other,
$g'$, satisfies $\rho(g')=\rho(h)^{m/2}$.
\end{lemma}

\begin{proof}
Since $O$ is a free $H$-orbit (Definition~\ref{def:freeorbit}), $\rho|_{H}$ is
injective and $\rho(H)=\langle
\rho(h)\rangle$ is generated by an $m$-cycle. As $G$ is abelian $\rho(k)$
commutes with $\rho(h)$; the centralizer of an $m$-cycle in $\Sym(O)$, where
$|O|=m$, is the cyclic group it generates (a full $m$-cycle is a regular
permutation, so its centralizer in $\Sym(O)$ has order $m$ and therefore equals
$\langle\rho(h)\rangle$). Hence $\rho(k)=\rho(h)^{s}$ for some $s$, so
$\operatorname{Im}\rho=\rho(H)$ and $|N|=|G|/|\rho(H)|=2m/m=2$.

Write $g$ for the non-identity element of $N$, so $N=\langle g\rangle=\{e,g\}$.
From $\rho(k)=\rho(h)^{s}$, the
element $k\,h^{-s}$ satisfies $\rho(k\,h^{-s})=\rho(h)^{s}\rho(h)^{-s}=\id$, so
$k\,h^{-s}\in N$. Moreover $k\,h^{-s}\neq e$, since $k\,h^{-s}=e$ would give
$k=h^{s}\in H$, contradicting $k\notin H$; as $N=\{e,g\}$, this forces
$g=k\,h^{-s}$, and then $g=k\,h^{-s}\notin H$ (again because $k\notin H$ while
$h^{-s}\in H$). As $g\in N$ and $|N|=2$, $g^{2}=e$; expanding,
$g^{2}=k^{2}h^{-2s}=h^{-2s}=e$ (using $k^{2}=e$ and $G$ abelian), forcing
$2s\equiv0\pmod{m}$, so $s\in\{0,m/2\}$ and $g\in\{k,h^{m/2}k\}$; and
$N=\ker\rho$ fixes $O$ pointwise. The
elements of order $2$ in
$G=\Z_{2^{n}}\times\Z_{2}$ are $h^{m/2}$, $k$, $h^{m/2}k$; those outside $H$ are
exactly $k$ and $h^{m/2}k$. One generates $N$; for the other, $g'$, we have
$\rho(g')\neq\id$, and $\rho(g')$ is a non-trivial involution commuting with the
$m$-cycle $\rho(h)$, hence (centralizer again, and
Lemma~\ref{lem:mcycle}(2)) $\rho(g')=\rho(h)^{m/2}$.
\end{proof}

We record how $K$ may permute free $H$-orbits.

\begin{lemma}\label{lem:korbit}
$H$ fixes every $H$-orbit setwise, so $G/H\cong K\cong\Z_{2}$ acts on the set of
free $H$-orbits. For a free $H$-orbit $O$, either $k$ fixes $O$ setwise, in
which case $O$ is $G$-invariant and Lemma~\ref{lem:kernel} applies, or $k$ maps
$O$ onto a
different free $H$-orbit $O'$, in which case $O\cup O'$ is a single $G$-orbit of
size $2m$ and is an antichain by Lemma~\ref{lem:antichain}.
\end{lemma}

\begin{proof}
$R$ is $G$-invariant (Remark~\ref{rem:invariance}) and $H$ acts trivially on
the set of $H$-orbits, so the $G$-action on free $H$-orbits factors through
$G/H\cong\Z_{2}$. If $k(O)=O$ then $G=\langle H,k\rangle$ preserves $O$. If
$k(O)=O'\neq O$ then for $x\in O$, $G\cdot x=H\cdot x\cup kH\cdot x=O\cup O'$, a
single $G$-orbit of size $2m$, which is an antichain by
Lemma~\ref{lem:antichain}.
\end{proof}

The next lemma makes explicit, once, the passage repeatedly used below from
``$g$ swaps $y,y'$ and fixes the rest'' to ``$y,y'$ have equal full
row-signatures''.

\begin{lemma}\label{lem:rowsig}
Let $O^{*}$ be an $H$-orbit of $P$ and $y\neq y'\in O^{*}$. Suppose there is
$g\in\Aut(P)$ with $g\cdot y=y'$ that fixes every point of $P\setminus O^{*}$.
Then $y$ and $y'$ have equal full row-signatures.
\end{lemma}

\begin{proof}
$P$ is the disjoint union $O^{*}\sqcup(P\setminus O^{*})$, so the full
row-signature of a point of $O^{*}$ is determined by its relations to
$O^{*}$ and to $P\setminus O^{*}$.

\emph{Relations to $O^{*}$.} By Lemma~\ref{lem:antichain} the $H$-orbit $O^{*}$
is an antichain, so each of $y$ and $y'$ is incomparable to every other point of
$O^{*}$ (and $y\parallel y'$). Thus $y$ and $y'$ have identical relations to all
of $O^{*}$.

\emph{Relations to $P\setminus O^{*}$.} Let $z\in P\setminus O^{*}$. Then
$g\cdot z=z$, and since $g$ is an order-automorphism with $g\cdot y=y'$,
\[
y\mathrel{\rho}z\iff g\cdot y\mathrel{\rho}g\cdot z\iff y'\mathrel{\rho}z
\qquad(\rho\in\{<,>,\parallel\}).
\]
Hence $y$ and $y'$ have identical relations to every $z\in P\setminus O^{*}$.

Combining the two parts, $y$ and $y'$ have equal full row-signatures.
\end{proof}

\begin{lemma}\label{lem:uniform}
Suppose $|Q|=1$, say $Q=\{q\}$. Then every element of $G$ fixes $q$, and for
any subgroup $\Gamma\le G$ and any $\Gamma$-orbit $T\subseteq P$, all points of
$T$ have the same relation to $q$ (all $<q$, all $>q$, or all $\parallel q$).
\end{lemma}

\begin{proof}
$Q$ is $G$-invariant (Remark~\ref{rem:invariance}); being a singleton it is
fixed pointwise, so $g\cdot q=q$ for every $g\in G$. Fix $x\in T$; since $T$ is a
$\Gamma$-orbit we have $T=\Gamma\cdot x=\{\gamma\cdot x:\gamma\in\Gamma\}$.
Because $q$ is $\Gamma$-fixed, its orbit is $\{q\}$, so the orbit $T$ either
equals $\{q\}$ (whereupon the claim is trivial) or avoids $q$; assume the
latter, so $x\neq q$. Then exactly one of $x<q$, $x>q$, $x\parallel q$ holds
(trichotomy for distinct points of a poset); let $\rho$ be that relation. For
any $\gamma\in\Gamma$, $\gamma$ is an order-automorphism and $\gamma\cdot q=q$,
so
\[
x\mathrel{\rho}q\iff\gamma\cdot x\mathrel{\rho}\gamma\cdot q
\iff\gamma\cdot x\mathrel{\rho}q ,
\]
and since $x\mathrel{\rho}q$ holds, so does $\gamma\cdot x\mathrel{\rho}q$. As
$\gamma\cdot x$ ranges over all of $T$, every point of $T$ bears the relation
$\rho$ to $q$.
\end{proof}

In each step below we reach a contradiction in one of three uniform ways.
\textbf{(R)} Two distinct points of a free $H$-orbit acquire equal full
row-signatures, impossible by Corollary~\ref{cor:sigrigid}. \textbf{(A)} There
is an antichain $Y\subseteq P$ with $|Y|\ge m$ such that every point of $Y$
bears the same relation to each point of $P\setminus Y$; then every permutation
$\pi$ of $Y$ extends to an automorphism of $P$. Indeed, let $\varphi$ agree with
$\pi$ on $Y$ and fix $P\setminus Y$ pointwise. To see $\varphi$ preserves
$\le$, take $u\le v$: if $u,v\in Y$ then $u=v$ (an antichain has no strict
relations), so $\varphi(u)=\varphi(v)$; if $u,v\in P\setminus Y$ then
$\varphi$ fixes both; if $u\in Y$, $v\in P\setminus Y$, then by hypothesis every
point of $Y$, in particular $\pi(u)$, bears the same relation to $v$ as $u$ does,
so the relation is preserved; if $u\in P\setminus Y$, $v\in Y$, then by
hypothesis every point of $Y$, in particular $\pi(v)$, bears the same relation to
$u$ as $v$ does, so the relation is preserved. The same
verification applied to $\pi^{-1}$ shows $\varphi^{-1}$ also preserves $\le$;
hence $\varphi$ is an order-automorphism. Distinct $\pi$ give distinct
$\varphi$, so $|\Aut(P)|\ge|Y|!\ge m!>2m=|G|$ (the inequality $m!>2m$ holds for
every $m=2^{n}$ with $n\ge3$, already $8!>16$), whence $\Aut(P)\not\cong G$.
\textbf{(S)} A size estimate violates minimality.

\begin{lemma}\label{lem:Qge2}
Let $n\ge3$ and let $P$ be a minimal poset realizing $G=\Z_{2}\times\Z_{2^{n}}$.
Then $|Q|\ge2$.
\end{lemma}

\begin{proof}
Let $r$ be the number of free $H$-orbits, so $|R|=rm$; by
Lemma~\ref{lem:freeexists}, $r\ge1$, and by minimality with
Lemma~\ref{lem:upper}, $|P|\le2^{\,n+1}+2$. Assume $|Q|\le1$. The argument is
exhaustive: by the size reduction $r\in\{1,2\}$, and for each value of $r$ the
group $K=\langle k\rangle$ either fixes every free $H$-orbit setwise or swaps
two of them (Lemma~\ref{lem:korbit}); we show each of the resulting
configurations, with $|Q|\in\{0,1\}$, falls into mode \textbf{(R)},
\textbf{(A)}, or \textbf{(S)}.

\emph{Size reduction.} If $r\ge3$ then
$|P|\ge rm\ge3\cdot2^{n}>2^{\,n+1}+2$ (the gap is $2^{n}-2\ge6$), so by
\textbf{(S)} we may assume $r\in\{1,2\}$.

\emph{The non-free point.} If $|Q|=1$ we write $Q=\{q\}$ and use
Lemma~\ref{lem:uniform}: every element of $G$ fixes $q$, and the points of any
single $H$- or $G$-orbit all bear the same relation to $q$.

\medskip
\noindent\textbf{Case $r=1$.} Let $O$ be the unique free $H$-orbit; $k$ must fix
$O$ setwise (Lemma~\ref{lem:korbit}), so $O$ is $G$-invariant.
\emph{If $|Q|=0$}: then $P=O$ (the unique $H$-orbit), a single $G$-orbit, hence
an antichain by Lemma~\ref{lem:antichain}; take $Y=P$ ($|Y|=m\ge8$), where
$P\setminus Y=\varnothing$ makes the uniformity hypothesis of \textbf{(A)}
vacuous, giving mode \textbf{(A)}.
\emph{If $|Q|=1$}: by Lemma~\ref{lem:kernel} pick the involution
$g'\in G\setminus H$ with $\rho_{O}(g')=\rho_{O}(h)^{m/2}$, which by
Lemma~\ref{lem:mcycle}(1) is a fixed-point-free involution of $O$; let
$\{y,y'\}$ be one of its $2$-cycles, so $g'\cdot y=y'$ with $y\neq y'\in O$.
Here $P\setminus O=\{q\}$ and $g'$ fixes $q$ (every element of $G$ fixes the
singleton $Q$), so $g'$ fixes $P\setminus O$ pointwise. By
Lemma~\ref{lem:rowsig}, $y$ and $y'$ have equal full row-signatures; as $O$ is
a free $H$-orbit this contradicts Corollary~\ref{cor:sigrigid}, giving mode
\textbf{(R)}.

\medskip
\noindent\textbf{Case $r=2$.} Let the free $H$-orbits be $O_{1},O_{2}$.

\emph{Sub-case $k$ swaps $O_{1},O_{2}$.} By Lemma~\ref{lem:korbit}, $W:=O_{1}\cup
O_{2}$ is a single $G$-orbit of size $2m$ and an antichain. If $|Q|=0$ then
$P=W$; take $Y=P$ ($|Y|=2m\ge8$, $P\setminus Y=\varnothing$), giving mode
\textbf{(A)}. If $|Q|=1$ then
$P=\{q\}\cup W$; by Lemma~\ref{lem:uniform} (applied to the $G$-orbit $W$) all
points of $W$ bear the same relation to $q$, and $W$ is an antichain, so every
permutation of $W$ extends (fixing $q$) to an automorphism of $P$; take $Y=W$
($|Y|=2m\ge8$), giving mode \textbf{(A)}.

\emph{Sub-case $k$ fixes $O_{1},O_{2}$ setwise.} Then each $O_{i}$ is
$G$-invariant; let $\rho_{i}:G\to\Sym(O_{i})$, $N_{i}=\ker\rho_{i}$. By
Lemma~\ref{lem:kernel}, $|N_{i}|=2$ and $N_{i}\not\subseteq H$. If
$N_{1}=N_{2}$, the non-trivial common element fixes $O_{1},O_{2}$ pointwise (and
$q$, if present), hence all of $P$, contradicting faithfulness; so
$N_{1}\neq N_{2}$. Pick $g\in N_{1}\setminus\{e\}$. Since $g\notin N_{2}$,
$\rho_{2}(g)\neq\id$ is a non-trivial involution commuting with the $m$-cycle
$\rho_{2}(h)$, so by Lemmas~\ref{lem:kernel} and~\ref{lem:mcycle},
$\rho_{2}(g)=\rho_{2}(h)^{m/2}$, a fixed-point-free involution of $O_{2}$; let
$\{y,y'\}$ be one of its $2$-cycles, so $g\cdot y=y'$ with $y\neq y'\in O_{2}$.
Now $P\setminus O_{2}=O_{1}\,(\cup\{q\})$, and $g$ fixes $O_{1}$ pointwise
(as $g\in N_{1}=\ker\rho_{1}$) and fixes $q$ (every element of $G$ fixes $q$);
so $g$ fixes $P\setminus O_{2}$ pointwise. By Lemma~\ref{lem:rowsig}, $y$ and
$y'$ have equal full row-signatures; as $O_{2}$ is a free $H$-orbit this
contradicts Corollary~\ref{cor:sigrigid}, giving mode \textbf{(R)}.

\medskip
Every configuration with $|Q|\le1$ is impossible; hence $|Q|\ge2$.
\end{proof}

The following lemma isolates the key step of Proposition~\ref{prop:notsingle}:
the transport equation turns the $h$-action on a free orbit into a conjugate
$m$-cycle on its set of $Q$-signatures.

\begin{lemma}\label{lem:conj}
Let $O$ be a $G$-invariant free $H$-orbit, and suppose the map
$\iota:O\to\{<,>,\parallel\}^{Q}$, $\iota(x)=\Sig(x)$, is injective. Put
$S:=\iota(O)$, so $\iota:O\to S$ is a bijection and $|S|=m$. Then $S$ is
$\Phi(h_{Q})$-invariant. Write $h_{Q}|_{S}:=\Phi(h_{Q})|_{S}$ for the
restriction of $\Phi(h_{Q})$ to $S$ (Definition~\ref{def:sig}). Then
\[
\iota\circ\rho_{O}(h)\;=\;\bigl(h_{Q}|_{S}\bigr)\circ\iota .
\]
Consequently $h_{Q}|_{S}=\iota\circ\rho_{O}(h)\circ\iota^{-1}$ is conjugate to
$\rho_{O}(h)$; since $\rho_{O}(h)$ is an $m$-cycle, so is $h_{Q}|_{S}$, and
$\ord\bigl(h_{Q}|_{S}\bigr)=m$.
\end{lemma}

\begin{proof}
By the transport equation (Remark~\ref{rem:invariance}),
$\Sig(h\cdot x)=\Phi(h_{Q})\bigl(\Sig(x)\bigr)$ for every $x\in O$, i.e.\
$\iota(h\cdot x)=\Phi(h_{Q})\bigl(\iota(x)\bigr)$. As $x$ ranges over $O$, so does
$h\cdot x$, and $\iota$ is a bijection onto $S$; hence $\Phi(h_{Q})$ maps $S$ onto
$S$, and the displayed identity holds on all of $O$. Therefore
$h_{Q}|_{S}=\iota\circ\rho_{O}(h)\circ\iota^{-1}$. Conjugate permutations have
the same cycle type; $\rho_{O}(h)$ is an $m$-cycle because $O$ is a free
$H$-orbit of size $m$, so $h_{Q}|_{S}$ is an $m$-cycle and has order $m$.
\end{proof}

\begin{proposition}\label{prop:notsingle}
Let $n\ge3$ and let $P$ be a minimal poset realizing
$G=\Z_{2}\times\Z_{2^{n}}$. Then the union $R$ of free $H$-orbits does not
consist of a single free orbit.
\end{proposition}

\begin{proof}
Assume $R=O$ is one free $H$-orbit, $|O|=m$. Then $k$ fixes $O$ setwise (it is
the only free $H$-orbit), so $O$ is $G$-invariant; let $\rho:=\rho_{O}:G\to
\Sym(O)$ be the restriction of the action to $O$. As $O$ is a free $H$-orbit of
size $m$, $\rho(h)$ is an $m$-cycle (the only fact about $\rho$ used below). By
Lemma~\ref{lem:Qge2}, $|Q|\ge2$; in particular $Q\neq\varnothing$, and
$P=Q\cup O$.

\medskip
\noindent\textbf{Case (A): some distinct $x,x'\in O$ have
$\Sig(x)=\Sig(x')$.} Since $O$ is an antichain (Lemma~\ref{lem:antichain}) and
$P=Q\cup O$, the relations of $x$ to $P$ are: incomparability to every other
point of $O$ (in particular $x\parallel x'$), and $\Sig(x)$ on $Q$; likewise for
$x'$. As $\Sig(x)=\Sig(x')$, $x$ and $x'$ have equal full row-signatures,
contradicting Corollary~\ref{cor:sigrigid}.

\medskip
\noindent\textbf{Case (B): $x\mapsto\Sig(x)$ is injective on $O$.}
By Lemma~\ref{lem:conj} (applicable since $O$ is a $G$-invariant free
$H$-orbit), with $S=\{\Sig(x):x\in O\}$, the restriction $h_{Q}|_{S}$ is an
$m$-cycle, so
\begin{equation}\label{eq:ordm}
\ord\bigl(h_{Q}|_{S}\bigr)=m.
\end{equation}

We bound $\ord\bigl(h_{Q}|_{S}\bigr)$ from above in two steps; here it is
essential to distinguish $h_{Q}$ as a permutation of the \emph{set} $Q$ from
the permutation it induces on the \emph{signature space}
$\{<,>,\parallel\}^{Q}$.

\emph{Step 1 (the induced action on signatures is a homomorphic image).} By
Definition~\ref{def:sig}, $h_{Q}$ acts on the signature space
$\{<,>,\parallel\}^{Q}$ as $\Phi(h_{Q})$, where
$\Phi:\Sym(Q)\to\Sym\bigl(\{<,>,\parallel\}^{Q}\bigr)$,
$\Phi(\pi):\Sigma\mapsto\Sigma\circ\pi^{-1}$, is a group homomorphism. Hence
$\ord\bigl(\Phi(h_{Q})\bigr)$ divides $\ord(h_{Q})$, the order of $h_{Q}$ in
$\Sym(Q)$. Since $h_{Q}|_{S}=\Phi(h_{Q})|_{S}$ and $S$ is
$\Phi(h_{Q})$-invariant (Lemma~\ref{lem:conj}), for every $\ell$ one has
$\bigl(h_{Q}|_{S}\bigr)^{\ell}=\Phi(h_{Q})^{\ell}\big|_{S}$, so any $\ell$
annihilating $\Phi(h_{Q})$ also annihilates $h_{Q}|_{S}$; thus
$\ord\bigl(h_{Q}|_{S}\bigr)$ divides $\ord\bigl(\Phi(h_{Q})\bigr)$. Combining,
\begin{equation}\label{eq:orddiv}
\ord\bigl(h_{Q}|_{S}\bigr)\ \big|\ \ord(h_{Q})\qquad(\text{order in }\Sym(Q)).
\end{equation}

\emph{Step 2 (the order of $h_{Q}$ on $Q$ divides $m/2$).} Each $q\in Q$ lies in
a non-free $H$-orbit, so $\Stab_{H}(q)\neq\{e\}$; the cycle of $h_{Q}$ through
$q$ is the $\langle h\rangle$-orbit of $q$, of length $[H:\Stab_{H}(q)]$, which
divides $m=2^{n}$ and is $<m$ (as $\Stab_{H}(q)\neq\{e\}$). A divisor of $2^{n}$
that is $<2^{n}$ is $2^{j}$ with $j\le n-1$, hence divides $2^{n-1}=m/2$. Thus
every cycle length of $h_{Q}$ divides $m/2$, and $\ord(h_{Q})$, the least common
multiple of these lengths, divides $m/2$.

Combining Step 1 and Step 2 via~\eqref{eq:orddiv}, $\ord\bigl(h_{Q}|_{S}\bigr)$
divides $\ord(h_{Q})$, which divides $m/2<m$, contradicting~\eqref{eq:ordm}.

\medskip
Both cases are impossible, so $R$ is not a single free orbit.
\end{proof}

\begin{corollary}\label{cor:lower}
For $n\ge3$, every minimal poset $P$ realizing $\Z_{2}\times\Z_{2^{n}}$
satisfies
\[
|P|=|R|+|Q|\ \ge\ 2m+2\ =\ 2^{\,n+1}+2 .
\]
\end{corollary}

\begin{proof}
By Lemma~\ref{lem:freeexists} there is at least one free $H$-orbit, and by
Proposition~\ref{prop:notsingle} there is not exactly one; hence there are at
least two, giving $|R|\ge 2m=2^{\,n+1}$. By Lemma~\ref{lem:Qge2}, $|Q|\ge2$.
Adding gives the bound.
\end{proof}

\begin{theorem}\label{thm:main}
For $n\ge3$,\qquad $\beta(\Z_{2}\times\Z_{2^{n}})=2^{\,n+1}+2$.
\end{theorem}

\begin{proof}
Combine Corollary~\ref{cor:lower} (lower bound) with Lemma~\ref{lem:upper}
(upper bound).
\end{proof}

\begin{remark}\label{rem:n3}
This remark is explanatory and is not used elsewhere; the theorem and all proofs
assume $n\ge3$. Two of the numerical thresholds in the lower-bound argument are
self-contained and looser than $n\ge3$: Corollary~\ref{cor:sigrigid} needs only
$m\ge3$, and mode \textbf{(A)} needs only $m!>2m$, i.e.\ $m\ge4$. The size
reduction in Lemma~\ref{lem:Qge2}, by contrast, is \emph{not} self-contained:
it rules out $r\ge3$ by combining $|P|\ge3m$ with $|P|\le2m+2$, and the second
inequality is imported from Lemma~\ref{lem:upper}, which needs $n\ge3$ (it is
precisely the range in which $b(2^{n})=2$, so that $\beta(\Z_{2^{n}})=2^{\,n+1}$,
Corollary~\ref{cor:cyclic}). At $n=2$ one has $b(4)=3$ and $\beta(\Z_4)=12$, so
the correct bound is $|P|\le\beta(\Z_4)+2=14$ (Proposition~\ref{prop:timesZ2}),
and $3m=12$ is no longer greater than $14$; since $14$ is in fact the exact
value of $\beta(\Z_2\times\Z_4)$ \cite{DasMawiongZ4Pre}, not merely a loose
estimate, no sharper bound can repair the comparison. Thus the size-reduction
step, and with it Lemma~\ref{lem:Qge2} and Corollary~\ref{cor:lower}, are not
established at $n=2$ by this argument; this is why $n=2$ is genuinely
exceptional and is treated separately, see the discussion in the Introduction.
\end{remark}

\section{Conclusion}

The argument provides a framework for minimal poset realizations: control the
kernel of the action on each free orbit, force enough free orbits via the
signature/transport machinery, and bound the non-free part. Natural next steps
are non-abelian groups and products involving higher prime powers, the
uniqueness of minimal realizations up to isomorphism, and algorithmic
construction or detection of minimal realizations.

\section*{Acknowledgement}
The authors thank the anonymous reviewer for the careful reading of the
manuscript and the helpful comments.

\end{document}